\documentclass[]{siamart1116}
\usepackage{enumitem}
\usepackage{amssymb}
\usepackage{amsmath}
\usepackage[a4paper,margin=1in]{geometry}
\usepackage{graphicx}
\usepackage{tikz}

\title{Leveraging the Cayley Hamilton Theorem For Efficiently Solving The Jordan Canonical Form Problem}
\author{Author: Lloyd Nesbitt \\ email: xyz333math@gmail.com}

\begin{document}

	\maketitle

	\begin{abstract}
		Given an $n \times n$ nonsingular matrix and it's characteristic polynomial as the starting point, we will leverage the Cayley-Hamilton Theorem to efficiently calculate the maximal length Jordan Chains for each distinct eigenvalue of the matrix.  Efficiency and speed are gained by seeking a certain type of starting vector as the 1st step of the algorithm.  The method for finding this starting vector does not require calculating the $ker(A-\lambda I)^k$ which is quite an expensive operation, and which is the usual approach taken in solving the Jordan Canonical basis problem.  Given this starting vector, all remaining vectors in the Jordan Chain are calculated very quickly in a loop.  The vectors comprising the Jordan Chains will then be used to minimize the amount of equation solving in order to find the remaining generalized eigenvector basis.  We will prove a theorem that justifies why the resulting Jordan Chains are of maximal length.  We will also justify how to derive the starting vector which will subsequently be used to calculate the Maximal Jordan Chains.
	\end{abstract}

	\section{Introduction}
	We present an algorithm and a set of techniques for solving the Jordan Canonical Form problem that is beneficial in two ways:
	\begin{enumerate}[label=\arabic*.]
		\item enhancing and improving the experience of learning Jordan Canonical Form problem solving in a pedagogical setting and
		\item providing an efficient algorithm to solve for a Jordan Canonical basis for both manual problem solving as well as implementations in software.
	\end{enumerate}

	It is nearly universally agreed that understanding Jordan Canonical Form, from a theoretical point of view, as well as a practical step by step approach to solving problems, is an essential part of a college math student's curriculum and training.  There is good reason for that as the Jordan Canonical Form problem greatly enhances one's grasp of many important fundamental concepts in Linear Algebra.
	
	Many standard math text books contain a description of the steps required for manually solving the Jordan Canonical Form problem \cite{axler2015linear} \cite{hungerford1974verlag} \cite{lipschutz1968schaum} \cite{neringevard}.  We claim that the algorithm presented here offers a novel approach to solving Jordan Canonical Form problems which strikes a balance of both efficiency and simplicity.
	
	In a classroom setting, for reasons of practicality, solving Jordan Canonical Form problems manually in a reasonable amount of time  necessitates limiting the matrix size to 4x4 or less due to the intermediate calculations involved. Even a 4x4 is almost too big as the act of squaring a 4x4 matrix manually is very arduous requiring over 100 combined multiplications and additions.  This becomes an issue with Jordan Canonical Form problems: smaller sized matrices are too trivial to cover all aspects of Jordan Canonical problem solving techniques and larger matrices become too unwieldy to do manually.  
	
	We make the following claims about the algorithm presented herein.
	\begin{enumerate}[label=\arabic*.]
	\item It greatly reduces the amount of tedious calculations involved, thus opening up a wider range Jordan Canonical Form problems that can be reasonably assigned to students for homework and tests.
	\item It guides the problem solver to think about and focus on the most important concepts of Jordan Canonical Form problems.  In particular, emphasis is placed on the generalized Eigenspaces and their basis consisting of Jordan Chains.
	\end{enumerate}
	  
	The novelty of our method and the most important part is what we refer to as steps A.1 and A.2 below.  These steps are the first steps performed and the outcome is that all the maximal length Jordan Chains get calculated quickly and without solving any equations.  Furthermore, for matrices that have many larger Jordan blocks relative to Jordan blocks of size one, this outcome constitutes the bulk of the necessary calculations needed to completely solve the Jordan Normal Form problem.  These first two steps are justified by Theorem 5, which we refer to as the "Main Theorem", and Theorem 6 given below.  The remaining steps of the algorithm are a well orchestrated usage of well known and well understood ideas.  A very brief mention of the core idea of this paper is made on mathoverflow.net \cite{mathoverflowcayleyhamilton} which points out how the Cayley-Hamilton theorem comes into play and how eigenvectors can be found without solving any equations.  We have taken those core ideas and greatly expanded the details and provided important justification via Theorem 5 and 6 below, as well as provided an algorithm to completely solve for a Jordan Canonical basis.

	This algorithm is designed specifically to solve non-diagonalizable matrices.  For linear operators that are diagonalizable, it will not offer any advantages over other known diagonalization algorithms.  The more complex and the larger the generalized eigen spaces are the better this algorithm will perform.  Or to rephrase, if the exponents of the linear factors of the minimal polynomial are strictly greater than 1, then this algorithm will excel in that situation.  The larger the exponents, the better the performance.
	
	\begin{keywords}
		Jordan Normal Form, Jordan Canonical Form, Generalized Eigen Space, Jordan Chain, Cayley Hamilton
	\end{keywords}

	\section{Overview}
	Given a non-singular linear transformation over a vector space V represented by an $n \times n$ matrix A over the field of complex numbers and it's characteristic polynomial in factored form $(x-\lambda_{1})^{j_1}(x-\lambda_{2})^{j_2}\cdots(x-\lambda_{m})^{j_m}$, the goal is to efficiently solve the "Jordan Canonical Form" by starting with vectors in the generalized eigenspace that are as far away from the eigenvectors as possible.  By "far away" we mean that we seek a generalized eigenvector which is in the kernel of $(A - \lambda_k I)^r$ but not in the kernel of $(A - \lambda_k I)^{r-1}$ where $r$ is the greatest integer such that $r \le j_k$.  If $(r \ge 2)$ and we denote the starting vector as $\overrightarrow{w}_r$ then a Jordan Chain corresponding to $\lambda_k$ is obtained by calculating $\overrightarrow{w}_{j-1} = (A - \lambda_k I)\overrightarrow{w}_j$ for $j=r \cdots 2$. Additionally we want to obtain the starting vectors as quickly as possible with minimal calculations.

	The process can be broken down into two major groups of steps.  The first group of steps is to find the maximal Jordan Chains of each distinct eigenvalue.  The second group of steps is to find all the smaller (non-maximal) Jordan Chains in a way that uses the basis vectors already found from first steps to minimize the calculations involved.

	\section{Notation and Terminology}
	$\textbf{A}\equiv$ $n \times n$ nonsingular matrix over $\mathbb{C}$.  Note that throughout this article we are always talking about the same linear transformation T with the representation of T by the $n \times n$ matrix A.  We may refer to A (rather than T) as the linear transformation.

	$\textbf{G}_A(\lambda_k)\equiv$ The generalized eigenspace of the linear tranformation A corresponding to $\lambda_k$.  I.e. it is the subspace of V consisting of all vectors $\overrightarrow{v} \in V$ such that $(A-\lambda_k I)^{j_k} \overrightarrow{v} = 0$ for some integer $j_k$

	$\textbf{G}_A(\lambda_k, i)\equiv$ The kernel of $(A-\lambda_k I)^i$.

	\textbf{Generalized Eigenvector} $\equiv$ Any vector in V which is in $ker(A-\lambda_i I)^k$ for some eigenvalue $\lambda_i$ and some positive integer $k$.  Note that we include an eigenvector where $k=1$ as being a member of the set of generalized eigenvectors.

	\textbf{Jordan Chain} $\equiv$ an ordered sequence of vectors belonging to a single generalized eigen subspace $G_A(\lambda_i)$ relative to a linear operator with matrix A as follows: using a starting vector $\overrightarrow{v}_0 \in G_A(\lambda_i)$ and then obtaining a sequence of non-zero vectors $\overrightarrow{v}_0, (A-\lambda_i I)^1\overrightarrow{v}_0, \dots, (A-\lambda_i I)^{r-1}\overrightarrow{v}_0, (A-\lambda_i I)^r\overrightarrow{v}_0$ where $(A-\lambda_i I)^{r+1}\overrightarrow{v}_0=\overrightarrow{0}$. Note that to form the Jordan transformation matrix P, the sequence of vectors are placed in the reverse order as column vectors of P.

	\textbf{Maximal Jordan Chain} $\equiv$ \textbf{Maximal Length Jordan Chain} $\equiv$ A Jordan Chain as defined above that has maximal length for the given eigenvalue.  I.e. any other Jordan Chain for $\lambda_i$ must have length $\leq$ the length of a Maximal Jordan Chain for $\lambda_i$.

	$\textbf{P} \equiv$ a change of basis matrix such that $P^{-1} A P=J$ where $J$ is a matrix in Jordan Normal form.  In this paper we also use $P_i$ and $J_i$ to denote "intermediate" matrices, formed during the algorithm, which consist of block matrices in which part of the matrix $J_i$ consists of Jordan Blocks and part of the matrix $J_i$ is something other than a Jordan Block.  And similarly for $P_i$ part of the matrix is generalized eigenvectors as columns and part of the $P_i$ are columns other than generalized eigenvectors.

	$\overrightarrow{\textbf{e}}_i$ the standard basis vector.  E.g. in $\mathbb{C}^4$, $\overrightarrow{e}_3 = (0, 0, 1, 0)$.

	\section{Algorithm step by step}
	Starting assumptions:
	\begin{enumerate}[label=assumption \arabic*:]
		\item Given an $n \times n$ non-singular matrix A over $\mathbb{C}$ \item Given the characteristic polynomial $(x-\lambda_{1})^{j_1}(x-\lambda_{2})^{j_2}\cdots(x-\lambda_{m})^{j_m}$
		\item None of the generalized eigen vectors of A are standard basis vectors $\overrightarrow{e}_i$.  Note that this is a special case which can easily be dealt with, but it becomes a nusaiance to the main algorithm if we allow it.
	\end{enumerate}
	Then for each distinct eigenvalue $\lambda_i\in \{\lambda_1, \lambda_2,\cdots,\lambda_m \}$ perform the following steps:

	\begin{enumerate}[label=Step A.\arabic*:]
		\item choose a random vector $\overrightarrow{w} \in V$ with "calculation friendly" values.  A good choice is $\overrightarrow{w} = (1,1,\cdots,1)$ or $\overrightarrow{w}=(1, -1, 1, -1, \dots, (-1)^{n-1})$.  The important point about selecting $\overrightarrow{w}$ is that it has a non-zero coefficient of every basis vector.  See Theorem 6 below for justification as to why a purely random choice of vectors almost always works.  Note: if we are returning to this step from Step A.5.b below then the starting vector $\overrightarrow{w}$ must be linearly independent from any previous starting vectors for the given eigenvalue as well as linearly independent from all generalized eigen basis vectors previously found for the given eigenvalue $\lambda_i$.  See Step C.1 for details as to how to proceed in this scenario.
		\item for each $\lambda_k \neq \lambda_i$ remove the $G_A(\lambda_k)$ component from $\overrightarrow{w}$ via $\overrightarrow{w} = (A-\lambda_k I)^{j_k} \overrightarrow{w}$.  See implementation notes below.
		\item at the end of step 2 we will have a non-zero vector $\overrightarrow{w} \in G_A(\lambda_i)$ unless we made a bad initial guess for $\overrightarrow{w}$ that did not contain any component in $G_A(\lambda_i)$.  Calculate a Jordan Chain as follows: let $\overrightarrow{g}_0=\overrightarrow{w}, \overrightarrow{g}_1=(A-\lambda_i I)\overrightarrow{g}_0, \cdots, \overrightarrow{g}_h=(A-\lambda_i I)\overrightarrow{g}_{h-1}$ where h is an integer such that $\overrightarrow{g}_h$ is non-zero but $(A-\lambda_i I)\overrightarrow{g}_h = 0$.  Note: we will prove below that with a good starting vector $\overrightarrow{w}$ this is a maximal length Jordan Chain.
		\item save the Jordan Chain vectors from step 3 in a list $\mathcal{L}_i=\{\overrightarrow{g}_h, \cdots \overrightarrow{g}_0\}$ which we will be using later on to form the P matrix.  Note that we listed the vectors $\overrightarrow{g}_i$ in reverse order because that is the order they will be placed in the P matrix with the eigenvector on the left.  Now if $\lambda_i$ was not the last eigenvalue then move to the next $\lambda_{i+1}$ and go to step 1 and repeat all steps for $\lambda_{i+1}$.  Otherwise move on to step 5.
		\item Determine the next step among the following three choices: (i) we are finished.  (ii) return to Step A.1 or  (iii) go to Steps B.1 - B.5.  
		\begin{enumerate}[label=(\roman*)]
			\item If the Maximal Jordan blocks found so far span the entire space, then we have a complete generalized eigen basis and we are done.  Just form the P matrix with the basis vectors found and optionally compute $P^{-1}$.
			\item If the number of generalized eigen basis vectors found for the current eigenvalue $\lambda_i$ is less than or equal to half the exponent of the corresponding linear factor in the characteristic polynomial, then return to Step A.1 but take care to choose a different starting vector $\overrightarrow{w}$ that is linearly independent from the previously used $\overrightarrow{w}$ as well as all Jordan Basis vectors found so far for $\lambda_i$.  As an example, suppose 5 is an eigenvalue and the characteristic polynomial has a factor of $(x-5)^7$ and we have a Jordan Chain of length 3 found so far.  Then that means that there are 7 - 3 = 4 remaining generalized eigenvectors to be found corresponding to eigenvalue 5.  In particular there might be another maximal length Jordan Chain of length 3.  I.e. the sizes of the remaining Jordan Blocks could be any one of the following four possibilities: $ \{3, 1\}, \{2, 2\}, \{2, 1, 1\}, \{1, 1, 1, 1\} $.  We don't know which one of these four possibilities it will be.  But we should definitely try to find the possible block of size 3 if it exists by the technique of the A-Steps.  If there is not another block of size 3 then it will become apparant in step A.3.  See Step C. for details.
			\item If all the maximal length Jordan Chains have been found for all the eigenvalues $\lambda_i$ and if there are one or more $ i $ for which the basis vectors found so far do not span $G(\lambda_i)$, then we have more generalized eigenvectors to find for $\lambda_i$.  If that is the case then move to the next major part of the algorithm: Steps B.1 - B.4.
		\end{enumerate}
		
	\end{enumerate}

	\begin{enumerate}[label=Step B.\arabic*:]
		\item \textit{The goal of the remaining steps is to find all generalized eigen basis vectors corresponding to Jordan Blocks that have not yet been found above.  And we want to do that in an optimal way that leverages all the generalized eigen basis vectors found so far.}
		\newline
		Initialize an intermediate P-transition matrix $P_1$ as follows: form a Jordan transition matrix with the generalized eigenvectors found above as the $ 1^{st} $ r columns where r is the total number of vectors found above and $r < n$.  Then to complete an $n \times n$ matrix $P_1$ fill the remaining column vectors on the far right with n - r standard basis vectors so that the lower right hand block of $P_1$ is an identity matrix of size $(n-r) \times (n-r)$.  Also initialize a partial Jordan matrix $J_1$ with an upper left-hand block matrix $J_{B_1}$ which is in Jordan normal form corresponding to the partial Jordan basis vectors $\overrightarrow{g}_1,\cdots ,\overrightarrow{g}_r$.  And the remaining right-hand columns of $J_1$ are vectors which we must now solve for based on an equation which we will explain below.
		\newline
		what is going on?
		\begin{center}
			\begin{tikzpicture}[scale=0.9]
			\node at (0, 0.5) {\Large $P_1=$};
			\draw (0.7 + 0.1,0) -- (0.7 + 0,0) -- (0.7 + 0.0,1) -- (0.7 + 0.1,1);
			\node at (0.7 + 0.5,0.5) {\LARGE $\beta_1$};
			\draw[dashed] (0.7 + 1,0) -- (0.7 + 1,1);
			\node at (0.7 + 1.3, 0.75) {\Large 0};
			\draw[dashed] (0.7 + 1,0.4) -- (0.7 + 1.6,0.4);
			\node at (0.7 + 1.3, 0.16) {I};
			\draw (0.7 + 1.5,1) -- (0.7 + 1.6,1) -- (0.7 + 1.6,0) -- (0.7 + 1.5,0);

			\node at (0.7 + 1.8, 0) {,};
			\node at (3.1, 0.5) {\Large $J_1=$};

			\draw (3.7 + 0.1,0) -- (3.7 + 0,0) -- (3.7 + 0,1) -- (3.7 + 0.1,1);
			\node at (3.7 + 0.5,0.75) {\Large $J_{B_1}$};
			\draw[dashed] (3.7 + 1,0) -- (3.7 + 1,1);
			\node at (3.7 + 0.5, 0.2) {0};
			\draw[dashed] (3.7 + 0.1,0.4) -- (3.7 + 1.,0.4);
			\node at (3.7 + 1.3, 0.5) {$ U_1 $};
			\draw (3.7 + 1.5,1) -- (3.7 + 1.6,1) -- (3.7 + 1.6,0) -- (3.7 + 1.5,0);
			\end{tikzpicture}
		\end{center}
		$\beta_1=\overrightarrow{g}_1,\dots,\overrightarrow{g}_r$ where $\overrightarrow{g}_i$ are all the basis vectors found so far for all eigenvalues in steps A.1 through A.5 which are arranged in order to yield a Jordan block matrix $J_{B_1}$.  $U_1$ block is a list of column vectors $\overrightarrow{u}_1, \dots ,\overrightarrow{u}_{n-r}$ which we don't know and must solve for.  The equation we have is $A P_1 = P_1 J_1$.  By the structure of how we have constructed $P_1$ and $J_1$ it is straight forward to see that each $\overrightarrow{u}_k$ has an equation $\overrightarrow{a}_k = P_1 \overrightarrow{u}_k$ where $\overrightarrow{a}_k$ denotes the $k^{th}$ column of $A$ and $k$ ranges from $r+1$ to $n$.  We will use $LU-Decomposition$ of $P_1$ to solve for $u_k$.
		\item Choose a $\lambda_i$ which has the lowest exponent of those remaining unsolved Jordan Blocks.  We will be finding the next Jordan Chain by the traditional means of finding $ker(J_1 - \lambda_i)^s$ for the appropriate power $s$.  Note how we are using $J_1$ to find the kernel rather than $A$.  Then the generalized eigenvectors of $J_1$ will then be translated back to the corresponding generalized eigenvectors of A by matrix similarity transformation.  I.e. if $\overrightarrow{y}$ is a generalized eigenvector of $J_1$ then $P_1 \overrightarrow{y}$ is a generalized eigenvector of $A$.  And the reason we are using $J_1$ rather than $A$ at this stage is because $J_1$ is much more sparse than $A$ and thus we can raise it to a power much quicker than $A$. Furthermore that is the reason we choose a lower exponent mentioned above so that we can raise $(J_1 - \lambda_i)^s$ to as low a power as possible.  That way when we do finally get to the remaining unsolved Jordan Blocks with higher powers our intermediate $J_i$ matrix will be as sparse as possible.

		\item \textbf{Iteration step}. The above two steps yielded one additional Jordan Block.  The remaining steps are to repeat the above two steps finding one Jordan Block at a time until we are done.  And we want to take advantage of any work done up to this point to avoid re-work as much as possible.  One way to achieve that that we haven't yet mentioned is to implement LU Decomposition directly step by step rather than use a canned pre-built LU routine.  And we will implement LU Decomposition in such a way that each subsequent $i+1$ iteration uses the LU Decomposition of iteration $i$ as a starting point.  The nature of LU Decomposition easily lends itself to such custom implementations.  In this case we are LU-factoring $P_i$ at each iteration.  The difference between $P_{i+1}$ and $P_i$ is that a few generalized eigenvectors replace the corresponding columns containing standard basis vectors $\overrightarrow{e}_i$.  So the LU decomposition at each iteration is exactly the same as the previous iteration up to the column in $P_i$ where the generalized eigenvectors stop and the standard basis vectors begin.  So at each iteration we save the LU decomposition steps up to that point.  We will explain the details of this process below.  The following gives a pictorial overview of how the Step B. iteration works.

		\begin{tikzpicture}[scale=0.85]
		\node at (0, 0.5) {\Large $P_1=$};
		\draw (0.7 + 0.1,0) -- (0.7 + 0,0) -- (0.7 + 0.0,1) -- (0.7 + 0.1,1);
		\node at (0.7 + 0.5,0.5) {\LARGE $\beta_1$};
		\draw[dashed] (0.7 + 1,0) -- (0.7 + 1,1);
		\node at (0.7 + 1.3, 0.75) {\Large 0};
		\draw[dashed] (0.7 + 1,0.4) -- (0.7 + 1.6,0.4);
		\node at (0.7 + 1.3, 0.16) {I};
		\draw (0.7 + 1.5,1) -- (0.7 + 1.6,1) -- (0.7 + 1.6,0) -- (0.7 + 1.5,0);

		\node at (0.7 + 1.8, 0) {,};
		\node at (3.1, 0.5) {\Large $J_1=$};

		\draw (3.7 + 0.1,0) -- (3.7 + 0,0) -- (3.7 + 0,1) -- (3.7 + 0.1,1);
		\node at (3.7 + 0.5,0.75) {\Large $J_{B_1}$};
		\draw[dashed] (3.7 + 1,0) -- (3.7 + 1,1);
		\node at (3.7 + 0.5, 0.2) {0};
		\draw[dashed] (3.7 + 0.1,0.4) -- (3.7 + 1.,0.4);
		\node at (3.7 + 1.3, 0.5) {$ U_1 $};
		\draw (3.7 + 1.5,1) -- (3.7 + 1.6,1) -- (3.7 + 1.6,0) -- (3.7 + 1.5,0);

		\node at (5.9, .5) {\Large{$\dots$}};
		\node at (6.8, 0.5) {\Large $P_i=$};

		\draw (4.4 + 3.1,0) -- (4.4 + 3,0) -- (4.4 + 3,1) -- (4.4 + 3.1,1);
		\node at (4.4 + 4.3,0.5) {\LARGE $\beta_1 \dots \beta_i$};
		\draw[dashed] (4.4 + 5.6,0) -- (4.4 + 5.6,1);
		\node at (4.4 + 5.9, 0.75) {\Large 0};
		\draw[dashed] (4.4 + 5.6,0.4) -- (4.4 + 6.2,0.4);
		\node at (4.4 + 5.9, 0.16) {I};
		\draw (4.4 + 6.1,1) -- (4.4 + 6.2,1) -- (4.4 + 6.2,0) -- (4.4 + 6.1,0);
		\node at (10.8, 0) {,};

		\node at (11.3, 0.5) {\Large $J_i=$};
		\draw (12.0 + 0.1,0) -- (12.0 + 0,0) -- (12.0 + 0,1) -- (12.0 + 0.1,1);
		\node at (12.0 + 0.5,0.75) {\Large $J_{B_i}$};
		\draw[dashed] (12.0 + 1,0) -- (12.0 + 1,1);
		\node at (12.0 + 0.5, 0.2) {0};
		\draw[dashed] (12.0 + 0.1,0.4) -- (12.0 + 1.,0.4);
		\node at (12.0 + 1.3, 0.5) {$ U_i $};
		\draw (12.0 + 1.5,1) -- (12.0 + 1.6,1) -- (12.0 + 1.6,0) -- (12.0 + 1.5,0);
		\node at (14.6, 0.5) {\Large $\dots P,J$};
		\end{tikzpicture}
		\newline
		Where each $\beta_i$ represents one or more column vectors.  At each iteration the size of the Jordan Block Matrix $J_{B_i}$ grows by one Jordan block and the number of column vectors in $U_i$ decreases by an equal number of vectors that are in the Jordan block.
		\item At the end of above iteration we have our final complete transition matrix P and our final Jordan form matrix J.  Now complete the LU decomposition of P the same way that it's done at the beginning of each iteration.  Then we can use that LU decomposition to quickly obtain $P^{-1}$.
	\end{enumerate}

	\begin{enumerate}[label=Step C.\arabic*:]
		\item This step will only be performed in the case where there might be multiple Maximal Length Jordan Chains.  I.e. when there are two or more Jordan Chains of the same length which is also the maximal length of all Jordan Chains for the given eigenvalue.  So this step starts off at the point where we have already found at least one maximal length Jordan Chain for $\lambda_i$ and we are now trying to find an additional maximal length Jordan Chain if it exists.  It is guaranteed that there will be at least one maximal length Jordan Chain for any given eigenvalue $\lambda_i$ but it is not guaranteed that there will be more than one.  I.e. this is a trial and error step.  Thus lets suppose that corresponding to $\lambda_i$, the linear factor in the characteristic polynomial is $(x-\lambda_i)^{j_i}$ and we have found one Maximal Length Jordan Chain of length $m+1$ where $m+1 \leq j_i/2$.  Suppose that the starting vector to find the 1st maximal length Jordan Chain was $\overrightarrow{w}=(1,1,\dots,1)$ and that produced generalized eigen basis vectors $\overrightarrow\beta_0, \dots, \overrightarrow\beta_m$.  Then choose another starting vector $\overrightarrow{y}=(1,-1,1,-1,\dots,(-1)^{n-1})$.  Then $\overrightarrow{y}$ is clearly linearly independent of $\overrightarrow{w}$ and it is very likely that it is also linearly independent of $\overrightarrow\beta_0, \dots, \overrightarrow\beta_m$ as a result of Theorem 6 below.  If it is not then just use some common sense to choose a different vector.  Now using this starting vector $\overrightarrow{y}$ proceed with Steps A.2 and A.3.  At the end of Step A.3 we must check if all the vectors in the Jordan Chain produced in Step A.3 $\overrightarrow{g}_0, \overrightarrow{g}_1, \dots , \overrightarrow{g}_h$ are linearly independent from $\overrightarrow\beta_0, \dots, \overrightarrow\beta_m$.  Note also that we should get $h = m$.  If they are linearly independent then that implies that there is an additional maximal length Jordan Chain and we have successfully found it.  However if they are all linearly dependent to $\overrightarrow\beta_0, \dots, \overrightarrow\beta_m$, then it is very unlikely that there are any more maximal length Jordan Chains for $\lambda_i$, in which case we go to the "B-Steps" because the "B-Steps" are guaranteed to find all remaining Jordan Blocks for $\lambda_i$ whether they are maximal length or not.  Optionally, depending on the environment that we are performing this algorithm, an additional different starting vector $\overrightarrow{y}$ can be tried just to be more certain that we didn't make an unlucky guess on the 1st try.  Note that this Step C. will be repeated as long as $j_i - b_i \geq Jc_i$ where $Jc_i$ denotes the length of the Maximal Jordan Chain of $\lambda_i$ and $b_i$ denotes the total number of generalized eigen basis vectors found so far corresponding to $\lambda_i$, keeping in mind that we stop repeating this step as soon as we know there are no more Maximal Length Jordan Chains to be found.
	\end{enumerate}

	\section{Implementation Notes}
	In order to optimally calculate $(A-\lambda_k I)^{j_k} \overrightarrow{w}$, one should avoid raising the matrix to a power $(A-\lambda_k I)^{j_k}$ which is quite an expensive calculation (i.e. we would be performing an $(n \times n) \times (n \times n)$ matrix multiplication $j_k$ times).  Instead perform $\overrightarrow{w} = (A-\lambda_k I)\overrightarrow{w}$ in a loop from $1 \dots j_k$.

	\textbf{LU Decomposition of $P_i$}.  At iteration step i+1, the change from $P_i$ to $P_{i+1}$ looks like this: \textit{by way of example} suppose iteration i+1 yields an additional Jordan block of size 2 with basis vectors $\overrightarrow{h}_1$ and $\overrightarrow{h}_2$ and $P_i=[\overrightarrow{g}_1 \dots \overrightarrow{g}_{r_i} \overrightarrow{e}_{r_i+1} \dots \overrightarrow{e}_n]$.  Then $P_{i+1}=[\overrightarrow{g}_1 \dots \overrightarrow{g}_{r_i} \overrightarrow{h}_1 \overrightarrow{h}_2 \overrightarrow{e}_{r_i+3} \dots \overrightarrow{e}_n]$.  Let $L_i U_i=P_i$ be the LU Decomposition at step i which is saved in memory.  The upper matrix can be thought of as $E P_i$ where $E$ is a product of elementary row operations $E=E_1 \dots E_i$ \cite{neringevard}.  To obtain $U_{i+1}$ initialize a matrix $Q=U_i$, next let $y_1=E\overrightarrow{h}_1$, $y_2=E\overrightarrow{h}_2$.  Then replace columns $r_i+1$ and $r_i+2$ in $Q$ with $y_1$ and $y_2$.  Now do LU Decomposition of Q to get $Q_L Q_U=Q$.  Then set $U_{i+1}=Q_U$.  Set $L_{i+1}=L_i$ and then modify the columns $r_i+1$ and $r_i+2$ of $L_{i+1}$ setting them to those same two columns of $Q_L$.  And that is all that has to be done to obtain an incremental LU Decomposition at step i+1.  No further decomposition steps have to be performed on the right hand side of $P_{i+1}$ because the columns are the tail end of standard basis vectors $\overrightarrow{e}_k$ which are already in the form required by LU decomposition.  In otherwords think of this as doing LU decomposition as one would normally do, only we are starting in the "middle" of a decomposition that has already been done up to column $r_i$ and then to finish the LU factorization, we only have to do the decomp steps on the two new columns changed in $P_{i+1}$.

	Special case when a standard basis vector $\overrightarrow{e_i}$ is an eigenvector of $A$:  the above algorithm precludes this in the assumptions about the matrix A.  However if there does happen to be one or more standard basis vectors that are eigenvectors, then we can still do the algorithm, and the only change that is needed is in steps B.1 and B.3 above, when forming the right-hand columns of $P_i$, just use a different standard basis vector in place of the eigenvector.  Then $P_i$ will no longer have an identity matrix in the lower right-hand block, but the steps of the algorithm will still work the same way; in that case the matrix equation to solve for the $U_i$ column vectors will be matched with whatever column vector $a_i$ of $A$ matches the standard basis vector $e_i$ of $P_i$ that is used.

	\section{Theorems}
	\begin{lemma}
	If V is a vector space over $\mathbb{C}$ and A is an $n \times n$ matrix with distinct eigenvalues $\{\lambda_1,\cdots,\lambda_m\}$ Then\newline
	(a) V = $G_A(\lambda_1)\oplus\cdots\oplus G_A(\lambda_m)$
	\newline
	(b) each $G_A(\lambda_i)$ is invariant under A.
	\end{lemma}
	\begin{proof}
		See the proof in Sheldon Axler's book "Linear Algebra Done Right" 3rd ed. section 8.B. theorem 8.21.
		\cite{axler2015linear}
	\end{proof}
	\begin{lemma}
		If a subspace W of V is invariant under a linear operator A then W is invrariant under $(A-\lambda I)$ and W is invariant under $A^j$ for any positive integer j.
	\end{lemma}
	\begin{proof}
		suppose $\overrightarrow{w}\in W$ then $(A-\lambda I)\overrightarrow{w}$ = $A\overrightarrow{w} - \lambda \overrightarrow{w}$ = $\overrightarrow{w}_1 - \lambda \overrightarrow{w}\in W$ since $\overrightarrow{w}_1=A\overrightarrow{w}\in W$ and $\lambda \overrightarrow{w}\in W$.  To prove invariance under $A^j$ is quite straight forward: just look at $A^j \overrightarrow{w} = A^{j-1}A\overrightarrow{w}$.
	\end{proof} 
	\begin{lemma}
		$(A-\lambda I)$ commutes with $(A-\gamma I)$
	\end{lemma}
	\begin{proof}
		$(A-\lambda I) (A-\gamma I)$ = $A^2-\lambda A-\gamma A + \lambda \gamma I$ = $A^2-\gamma A-\lambda A + \gamma \lambda I$ = $(A-\gamma I) (A-\lambda I)$
	\end{proof}
	\begin{lemma}
		Suppose $\overrightarrow{u}\in G_A(\lambda_i, t)$ and $\overrightarrow{u}\notin G_A(\lambda_i, t-1)$ THEN $\forall$ $\lambda_h\neq \lambda_i, r\in \mathbb{Z}_{>0}$, $(A-\lambda_h I)^r \overrightarrow{u} \in G_A(\lambda_i,t)$ and $(A-\lambda_h I)^r \overrightarrow{u} \notin G_A(\lambda_i,t-1)$
	\end{lemma}
	\begin{proof}
		Since $(A-\lambda_h I)^r$ is invariant on each $G_A(\lambda)$ by Lemma 2 then it must map $\overrightarrow{u}$ into $G_A(\lambda_i,s)$ for some $s$.  We need to prove two things: first that $(A-\lambda_h I)$ cannot map $\overrightarrow{u}$ into a smaller subspace $G_A(\lambda_i,s)$ where $1 \le s<t$; and second it does map $\overrightarrow{u}$ into $G_A(\lambda_i,t)$. For the sake of contradiction lets make Assumption1 that $\exists \overrightarrow{u}\in G_A(\lambda_i, t)$ and $\overrightarrow{u}\notin G_A(\lambda_i, t-1)$ AND $(A-\lambda_h I)^r \overrightarrow{u} \in G_A(\lambda_i,s)$ for some $s < t$ and some $h\neq i$.  Then $(A-\lambda_h I)^r \overrightarrow{u} \in G_A(\lambda_i,s)$ $\Rightarrow$ $(A-\lambda_i I)^s (A-\lambda_h I)^r \overrightarrow{u} = \overrightarrow{0}$.  By Lemma 3 above the matrices commute and thus $(A-\lambda_i I)^s (A-\lambda_h I)^r \overrightarrow{u}$ = $(A-\lambda_h I)^r (A-\lambda_i I)^s \overrightarrow{u}$.  And since $s < t$ then $(A-\lambda_i I)^s \overrightarrow{u} \neq \overrightarrow{0}$ (this is by our assumption).  But that would mean that $\overrightarrow{z}=(A-\lambda_i I)^s \overrightarrow{u}$ is a non-zero vector that is contained in two distinct subspaces $G_A(\lambda_h)$ and $G_A(\lambda_i)$.  And that is impossible since V is a direct sum of the $G_A(\lambda_j)$ subspaces.  I.e. $G_A(\lambda_h)\cap G_A(\lambda_i) = \{\overrightarrow{0}\}$ for any $\lambda_h \neq\lambda_i$.  Thus Assumption1 leads to a contradiction.  That proves our first item.  The second item follows easily since the matrices $(A-\lambda_i I)^t$ and $(A-\lambda_h I)^r$ commute and thus $(A-\lambda_i I)^t (A-\lambda_h I)^r \overrightarrow{u}$ = $(A-\lambda_h I)^r (A-\lambda_i I)^t \overrightarrow{u}=(A-\lambda_h I)^r \overrightarrow{0}=\overrightarrow{0}$.
	\end{proof}
	\begin{theorem}[Main Theorem]
	Let $V=G_A(\lambda_1)\oplus G_A(\lambda_2)\oplus\cdots\oplus G_A(\lambda_m)$ be the direct sum of the generalized eigen spaces with respect to a linear transformation represented by an $n \times n$ matrix A.
	\newline
	Let $\overrightarrow{w}=\overrightarrow{w}_1 + \overrightarrow{w}_2 + \cdots + \overrightarrow{w}_m$ where
	\newline
	\newline
	$\overrightarrow{w}_1=a_{1,1} \overrightarrow{g}_{1,1} + \cdots + a_{1,j_1} \overrightarrow{g}_{1,j_1} \in G_A(\lambda_1,j_1)=G_A(\lambda_1)$.
	\newline
	$\vdots$
	\newline
	$\overrightarrow{w}_m=a_{m,1} \overrightarrow{g}_{m,1} + \cdots + a_{m,j_m} \overrightarrow{g}_{m,j_m} \in G_A(\lambda_m,j_m)=G_A(\lambda_m)$.
	\newline
	\newline
	where each $\{\overrightarrow{g}_{k1}, \overrightarrow{g}_{k2}, \cdots,\overrightarrow{g}_{kj_k} \}$ is a basis of the subspace $G_A(\lambda_k)$ which has dimension $j_k$.
	\newline
	And let $B=(A-\lambda_1 I)^{j_1} (A-\lambda_2)^{j_2}\cdots (A-\lambda_{k-1} I)^{j_{k-1}} (A-\lambda_{k+1} I)^{j_{k+1}} \cdots (A-\lambda_m I)^{j_m}$.
	\newline
	Then
	$B\overrightarrow{w} \in G_A(\lambda_k, j_k)$.  And if $t$ is the minimal positive integer such that $(A-\lambda_k)^t\overrightarrow{w}_k=0$ then $t$ is also the minimal integer such that $(A-\lambda_k)^t B\overrightarrow{w}=0$.  I.e. the matrix $B$ annihilates all components outside of the subspace $G_A(\lambda_k)$ and is invariant on the subspace $G_A(\lambda_k)$ and furthermore it does not change the minimum exponent $t$ for which $(A - \lambda_k)^t$ annihilates either of $\overrightarrow{w_k}$ or $B\overrightarrow{w}$.
	\end{theorem}
	\begin{proof}
		Let $\overrightarrow{w}$ be any arbitrary vector in V with a non-zero component in $G_A(\lambda_k)$.  Write $\overrightarrow{w}$ as a sum: $\overrightarrow{w}=\overrightarrow{w}_1 + \overrightarrow{w}_2 + \cdots + \overrightarrow{w}_m$ where each $\overrightarrow{w}_i \in G_A(\lambda_i)$. Now to compute $B \overrightarrow{w}$, observe that each factor $(A-\lambda_i I)^{j_i}$ will leave each subspace component invariant by Lemma 2.  And it will map the component with the same index to the 0-vector since the exponent equals the value of the exponent of the characteristic polynomial.  I.e. for $i\neq h$ and $\overrightarrow{w}_h\neq \overrightarrow{0}$, $(A-\lambda_i I)^{j_i} \overrightarrow{w}_h\in G_A(\lambda_h)$ and is non-zero. And $(A-\lambda_i I)^{j_i} \overrightarrow{w}_i=\overrightarrow{0}$.  And since B is the product of all the factors $(A-\lambda_i I)^{j_i}$ except for $i=k$ then $ B\overrightarrow{w}=\overrightarrow{z} \in G_A(\lambda_k)$.  I.e. it maps all components of $ \overrightarrow{w} $ to zero except for the component in the subspace $G_A(\lambda_k)$.  And by Lemma 4 the minimum exponent $t$ such that $\overrightarrow{w}_k \in G_A(\lambda_k, t)$ is equal to the minimum exponent $r$ such that $B\overrightarrow{w} \in G_A(\lambda_k, r)$.  This completes our proof of the main theorem.
	\end{proof}
	
	\begin{theorem}
		Let $\{\overrightarrow{v}_1, \overrightarrow{v}_2, \cdots, \overrightarrow{v}_n\}$ be any basis of V.  And let $\overrightarrow{w}$ = any non-zero vector in V chosen at random.  And let $\overrightarrow{w}=a_1 \overrightarrow{v}_1 + a_2 \overrightarrow{v}_2 + \cdots + a_n \overrightarrow{v}_n$ be the vector $\overrightarrow{w}$ expressed as a linear combination of the basis vectors.  Then the probability that $a_i = 0$ for any $i$ is zero.
	\end{theorem}
	\begin{proof}
		Choosing a random vector $ \overrightarrow{w} $ is the same as independently choosing each coefficient of the basis vectors one at a time and then adding up the components to obtain $\overrightarrow{w}$.  The probability of choosing $a_i=0$ is 0 because 0 is just one point of the entire complex plane and it is equally likely to choose any point in the complex plane.  And thus the probability of choosing a vector w such that at least one component has 0 coefficient is $prob(a_1=0 \vee a_2=0 \vee \cdots \vee a_n=0)\leq prob(a_1=0) + prob(a_2=0) + \cdots + prob(a_n=0)$ = $0 \times n = 0$.
	\end{proof}

	\section{Comments}
	If the algorithm is implemented in computer software then it should be done in the environment of arbitrary precision number fields that are implemented in symbolic algebra packages.  Solving Jordan Normal Form problems using fixed precision floating point numbers is generally avoided in the context of numerical analysis.  Eigenvalues are very sensitive to small perturbations caused by roundoff errors that completely alter the structure of the Jordan Normal Form of the matrix.  See reference "On the computation of Jordan Canonical Form" by Zhang and Zhang \cite{zhang2012computation}.

	In light of Theorem 6 above we can see that almost any choice of a starting vector $\overrightarrow{w}$ for the algorithm will do just fine.  Having said that, some common sense must be used in certain situations.  To illustrate with an extreme example: suppose the generalized eigen space basis was something like this: $ \overrightarrow{v}_1 $=(1,0,0,\dots,0), $ \overrightarrow{v}_2 $=(0,2,1,0,0,\dots,0), $ \overrightarrow{v}_3 $=(0,0,1,0,0,\dots,0), etc.  We can visually see that many of these basis vectors are in fact standard basis vectors such that for any given coordinate $i$ there is only one or two basis vectors that have a non-zero value in the $i^{th}$ component.  In this case we should not choose a vector $\overrightarrow{w}$ with any 0 coordinates.  A good rule of thumb is to choose w with non-zero components in every coordinate.  E.g. like $ \overrightarrow{w} $=(1,1,\dots,1).

	The reason we want $\overrightarrow{w}$ to have a non-zero coefficient of every generalized basis vector is because that is what yields the longest Jordan chain in the algorithm.  In otherwords for a given subspace $G_A(\lambda_i)$ the vector $\overrightarrow{w}$ needs to have a non-zero component of the subspace $G(\lambda_i,k)$ for the maximal integer $k$.  I.e. $G(\lambda_i,1) \subsetneq G(\lambda_i,2) \subsetneq \cdots \subsetneq G(\lambda_i,j_i)$ where $j_i$ is the maximal integer such that the subspace sequence is strictly increasing (i.e. $G(\lambda_i,k)=G(\lambda,j_i)$ $\forall k\geq j_i$).  And so we definitely want $\overrightarrow{w}$ to have a non-zero component in the subspace $G(\lambda_i,j_i)$ because otherwise the algorithm will not yields the longest Jordan Chain.
	
	Some of the ideas in this article came from the references as we will now describe.  The subspace spanned by a Jordan Chain is a cyclic subspace; this paper uses a similar technique to what we do in step A.3 to calculate a basis for a cyclic subspace: \cite{li1997determining}.  Attempts were made to derive the remaining smaller Jordan Chains (in the B-steps) by using starting vectors that are within the orthogonal complement of the subspace generated by basis vectors found so far in steps A.1 - A.5.  I.e. let $W$ be the subspace generated by all the basis vectors found in steps A.1 - A.5 and let $W^{\bot}$ be the orthogonal complement with respect to the standard inner product.  It was hoped that we would be able to calculate the remaining Jordan Chains directly as in step A.2 and A.3 by limiting the starting vector to $W^{\bot}$ and thus steering clear of $W$.  But every variation of this approach ended up yielding vectors in $W$ and not giving us any new basis vectors in $W^{\bot}$.  After many failed attempts it became apparent that there are subtleties involved with orthogonality in relation to linear operators \cite{BrianConrad}.  As Brian Conrad's web page points out, non-orthogonal linear operators do not respect inner products, and if you construct orthogonal vectors with an inner product, you cannot expect them to interact well with the linear operator, in relation to the orthogonal subspace.  Thus we abandoned that approach and explored a more direct block-matrix equation approach.  Ideas to progressively solve matrix equations via block-matrix fashion is here: \cite{bujosa1998jordan}. 
	\newline
	\section{Illustrative Example}
	Let
	\newline
	\newline
    $A=\begin{pmatrix}
     1679 & 5708 & -545 & 1814 & 948 & -1644 & 6250 & -672 & 5777 & -4718 \\
	-384 & -1320 & 99 & -450 & -217 & 376 & -1337 & 131 & -1274 & 1030 \\
	224 & 1301 & 692 & 1211 & -136 & -227 & -2023 & 550 & -1153 & 710 \\
	-152 & -530 & 16 & -211 & -103 & 145 & -451 & 28 & -501 & 355 \\
	-105 & -197 & 238 & 136 & -170 & 94 & -1170 & 206 & -971 & 635 \\
	-101 & -73 & 365 & 292 & -267 & 85 & -1652 & 304 & -1386 & 821 \\
	-55 & -20 & 245 & 225 & -127 & 49 & -1060 & 210 & -802 & 547 \\
	35 & 122 & -2 & 52 & 27 & -34 & 98 & -2 & 113 & -73 \\
	10 & -74 & -141 & -158 & 81 & -4 & 563 & -115 & 446 & -258 \\
	-5 & -76 & -75 & -100 & 36 & 8 & 273 & -60 & 205 & -113
	\end{pmatrix}
	$
	\newline
	\newline
	In order to to avoid visual clutter, we will not display the intermediate numeric values of the P matrix, and only display the structural parts of the intermediate matrices.
	\newline
	\newline
	The characteristic polynomial of $A$ is $(x-2)^4(x-3)^6$. During the algorithm steps, a check must be done that none of the standard basis vectors $\overrightarrow{e}_i$ nor the starting vector w of "all ones" are generalized eigenvectors.  Steps A.1 through A.5 proceed as follows:
	\newline
	\newline
	Let $\overrightarrow{w}=(1,1,\dots,1)$.
	\newline
	\newline	
	Calculate the max length Jordan Chain for eigenvalue 3: let $\overrightarrow{g}_0=(A-2 I)^4\overrightarrow{w}$ to remove $G_A(2)$ component from $\overrightarrow{w}$.  Theorem 5 tells us that $\overrightarrow{g}_0$ is in $G_A(3)$ and we can use it as the starting vector to calculate the Jordan Chain for eigenvalue 3.  Then $[\overrightarrow{g}_3, \overrightarrow{g}_2, \overrightarrow{g}_1, \overrightarrow{g}_0]=[(A-3 I)^3\overrightarrow{g}_0, (A-3 I)^2\overrightarrow{g}_0, (A-3 I)\overrightarrow{g}_0, \overrightarrow{g}_0]$ is our maximal length Jordan Chain for eigenvalue 3, since $(A-3 I)^4\overrightarrow{g}_0=\overrightarrow{0}$ and $(A-3 I)^3\overrightarrow{g}_0 \neq \overrightarrow{0}$.  Note that we will not know ahead of time how long the Jordan Chain will be.  Keep calculating in a loop until the result is the zero vector and then the eigen vector is the last non-zero vector calculated in the loop.
	\newline
	\newline
	Calculate the max length Jordan Chain for eigenvalue 2: let  $\overrightarrow{h}_0=(A-3 I)^4\overrightarrow{w}$ to remove $G_A(3)$ component from $\overrightarrow{w}$.  We only need an exponent of 4 rather than 6 because we now know (from the previous step) that the maximal length Jordan Chain of eigenvalue 3 is 4.  Then $[\overrightarrow{h}_2, \overrightarrow{h}_1, \overrightarrow{h}_0]=[(A-2 I)^2\overrightarrow{h}_0, (A-2 I)\overrightarrow{h}_0, \overrightarrow{h}_0]$ is our maximal length Jordan Chain for eigenvalue 2.
	\newline
	\newline
	We are now done with steps A.1 through A.5.  Move to the "B" steps.
	\newline
	\newline
	Let $J_1=\begin{pmatrix}
	3 & 1 & 0 & 0 & 0 & 0 & 0 & u_{1,8} & u_{1,9} & u_{1,10} \\
	0 & 3 & 1 & 0 & 0 & 0 & 0 & u_{2,8} & u_{2,9} & u_{2,10} \\
	0 & 0 & 3 & 1 & 0 & 0 & 0 & u_{3,8} & u_{3,9} & u_{3,10} \\
	0 & 0 & 0 & 3 & 0 & 0 & 0 & u_{4,8} & u_{4,9} & u_{4,10} \\
	0 & 0 & 0 & 0 & 2 & 1 & 0 & u_{5,8} & u_{5,9} & u_{5,10} \\
	0 & 0 & 0 & 0 & 0 & 2 & 1 & u_{6,8} & u_{6,9} & u_{6,10} \\
	0 & 0 & 0 & 0 & 0 & 0 & 2 & u_{7,8} & u_{7,9} & u_{7,10} \\
	0 & 0 & 0 & 0 & 0 & 0 & 0 & u_{8,8} & u_{8,9} & u_{8,10} \\
	0 & 0 & 0 & 0 & 0 & 0 & 0 & u_{9,8} & u_{9,9} & u_{9,10} \\
	0 & 0 & 0 & 0 & 0 & 0 & 0 & u_{10,8} & u_{10,9} & u_{10,10}
	\end{pmatrix}
	$
	\newline
	\newline
	Let $P_1=[\overrightarrow{g}_3,\overrightarrow{g}_2,\overrightarrow{g}_1,\overrightarrow{g}_0,\overrightarrow{h}_2,\overrightarrow{h}_1,\overrightarrow{h}_0,\overrightarrow{e}_8,\overrightarrow{e}_9,\overrightarrow{e}_{10}]$
	\newline
	Do LU factorization: $L_1U_1=P_1$.
	\newline
	Note that we have the equation $AP_1=P_1J_1$ and use the LU factors to solve for column vectors $\overrightarrow{u}_k$ via  $\overrightarrow{a}_k=P_1\overrightarrow{u}_k$ where $k$ is 8,9,10.  The equations for $\overrightarrow{u}_k$ hold because the right most column vectors of $P_1$ are standard basis vectors $e_k$.
	\newline
	The remaining Jordan blocks to be solved for have dimension 1 for eigenvalue 2 and either 2 or 1 for eigenvalue 3.  Therefore solve the smallest space first which is the 2-space.
	\newline
	Solve $ker(J_1 - 2I)$ to find the remaining eigenvector.  Be sure to chose a vector in the kernal which is linearly independent from the basis vector $\overrightarrow{h}_2$ already found above.  If we denote this vector as $\overrightarrow{c}$ then to get the eigenvector of A take $\overrightarrow{h}_3=P_1\overrightarrow{c}$.
	\newline
	Move to the next iteration (i=2).  Set the intermediate P matrix $P_2$ to be $P_1$ with the 8th column replaced by the eigenvector just obtained.  And set $ J_2 $ to be $J_1$ with the 8th column replaced with the corresponding vector of a Jordan Block.
	\newline
	Let $J_2=\begin{pmatrix}
	3 & 1 & 0 & 0 & 0 & 0 & 0 & 0 & u_{1,9} & u_{1,10} \\
	0 & 3 & 1 & 0 & 0 & 0 & 0 & 0 & u_{2,9} & u_{2,10} \\
	0 & 0 & 3 & 1 & 0 & 0 & 0 & 0 & u_{3,9} & u_{3,10} \\
	0 & 0 & 0 & 3 & 0 & 0 & 0 & 0 & u_{4,9} & u_{4,10} \\
	0 & 0 & 0 & 0 & 2 & 1 & 0 & 0 & u_{5,9} & u_{5,10} \\
	0 & 0 & 0 & 0 & 0 & 2 & 1 & 0 & u_{6,9} & u_{6,10} \\
	0 & 0 & 0 & 0 & 0 & 0 & 2 & 0 & u_{7,9} & u_{7,10} \\
	0 & 0 & 0 & 0 & 0 & 0 & 0 & 2 & u_{8,9} & u_{8,10} \\
	0 & 0 & 0 & 0 & 0 & 0 & 0 & 0 & u_{9,9} & u_{9,10} \\
	0 & 0 & 0 & 0 & 0 & 0 & 0 & 0 & u_{10,9} & u_{10,10}
	\end{pmatrix}
	$
	\newline
	\newline
	Let $P_2=[\overrightarrow{g}_3,\overrightarrow{g}_2,\overrightarrow{g}_1,\overrightarrow{g}_0,\overrightarrow{h}_2,\overrightarrow{h}_1,\overrightarrow{h}_0,\overrightarrow{h}_3,\overrightarrow{e}_9,\overrightarrow{e}_{10}]$.
	Incrementally update the $L_1, U_1$ factors, as discussed above, to obtain $L_2, U_2$ which incorporates new column vector $\overrightarrow{h}_3$.  Note that we're using $u_{i,j}$ only as variable names which are solved for new in each iteration.  I.e. don't confuse them with the $u_{i,j}$ column vectors found in previous iteration.  Solve for $u_{i,j}$ to complete the $J_2$ matrix.
	\newline
	Since we have found all the basis vectors in $G_A(2)$ then it is clear that the last two generalized eigenvectors to be found belong to $G_A(3)$.  Since we've already found one of them, then there are either two more eigenvectors or one eigenvector and one generalized eigenvector of eigenvalue 3 that remain to be found.  Thus compute $ker(J_2 - 3I)$ and determine if it has dimension 2 or 3.  Doing so we find that that the dimension of $ker(J_2 - 3I) = 2$, therefore there is only one eigenvector remaining and it's corresponding generalized eigenvector to be solved for.  Solve it via $ker(J_2 - 3I)^2$.  Find a nonzero vector in $ker(J_2 - 3I)^2$ which is linearly independent from $\overrightarrow{g}_3, \overrightarrow{g}_2, \overrightarrow{g}_1, \overrightarrow{g}_0$ and denote it as $\overrightarrow{y}_0$.  Then compute $\overrightarrow{y}_1=(J_2 - 3I)\overrightarrow{y}_0$.  Next take $\overrightarrow{z}_0=P_2 \overrightarrow{y}_0$ and $\overrightarrow{z}_1=P_2 \overrightarrow{y}_1$ to obtain the final two generalized eigen basis vectors. This will give us a complete basis of generalized eigenvectors.  With the last two vectors just found, create the final P matrix.  Then to obtain $P^{-1}$ incrementally update the LU factors $L_2$ and $U_2$ to get the LU factors of $P$.  From the LU factorization of $P$ we can quickly compute $P^{-1}$.  Then we will have
	\newline
	\newline
	$P^{-1}A P =
	\begin{pmatrix}
	3 & 1 & 0 & 0 & 0 & 0 & 0 & 0 & 0 & 0 \\
	0 & 3 & 1 & 0 & 0 & 0 & 0 & 0 & 0 & 0 \\
	0 & 0 & 3 & 1 & 0 & 0 & 0 & 0 & 0 & 0 \\
	0 & 0 & 0 & 3 & 0 & 0 & 0 & 0 & 0 & 0 \\
	0 & 0 & 0 & 0 & 2 & 1 & 0 & 0 & 0 & 0 \\
	0 & 0 & 0 & 0 & 0 & 2 & 1 & 0 & 0 & 0 \\
	0 & 0 & 0 & 0 & 0 & 0 & 2 & 0 & 0 & 0 \\
	0 & 0 & 0 & 0 & 0 & 0 & 0 & 2 & 0 & 0 \\
	0 & 0 & 0 & 0 & 0 & 0 & 0 & 0 & 3 & 1 \\
	0 & 0 & 0 & 0 & 0 & 0 & 0 & 0 & 0 & 3
	\end{pmatrix}
	$
	\newline
	\newline
	In summary for this particular matrix $A$ we were able to quickly calculate 7 out of 10 generalized eigen basis vectors without solving any equations or taking any matrix to a power.  All we needed was a sequence of $matrix \times vector$ calculations.  We only had to square a mostly sparse matrix once to find the last two basis vectors.  And we essentially only did one total LU factorization (done incrementally).  To find the remaining two non-maximal Jordan Chains, the number of calculations involved was greatly reduced by using intermediate mostly sparse partial Jordan block matrices and corresponding similarity matrix P rather than the original dense matrix A.

	\section{Acknowledgements} The author wishes to thank University of Texas at Arlington professor Michaela Vancliff for her help and guidance in putting together this article.
	
	\bibliographystyle{siamplain}
	\bibliography{jordPap}

\begin{thebibliography}{1}

\bibitem{axler2015linear}
{\sc S.~Axler}, {\em Linear Algebra Done Right}, Springer, 3rd~ed., 2015.

\bibitem{bujosa1998jordan}
{\sc A.~Bujosa, R.~Criado, and C.~Vega}, {\em Jordan normal form via elementary
  transformations}, SIAM review, 40 (1998), pp.~947--956.

\bibitem{BrianConrad}
{\sc B.~Conrad}, {\em Quotient vector spaces},
  \url{http://virtualmath1.stanford.edu/~conrad/diffgeomPage/handouts/qtvector.pdf}.

\bibitem{mathoverflowcayleyhamilton}
{\sc J.~Figueroa-O'Farrill}, {\em Applications of cayley hamilton theorem},
  \url{https://mathoverflow.net/questions/232132/applications-of-the-cayley-hamilton-theorem}.
\newblock section 14.

\bibitem{hungerford1974verlag}
{\sc T.~W. Hungerford}, {\em Algebra}, Springer-Verlag, 1974.

\bibitem{li1997determining}
{\sc T.~Li, Z.~Zhang, and T.~Wang}, {\em Determining the structure of the
  jordan normal form of a matrix by symbolic computation}, Linear algebra and
  its applications, 252 (1997), pp.~221--259.

\bibitem{lipschutz1968schaum}
{\sc S.~Lipschutz}, {\em Schaum's Outline of Theory and Problems of Linear
  Algebra}, McGraw-Hill, 1968.

\bibitem{neringevard}
{\sc E.~D. Nering}, {\em Linear Algebra and Matrix Theory}, John Wiley and
  Sons, 2nd~ed., 1974.

\bibitem{zhang2012computation}
{\sc Z.-N. Zhang and J.-N. Zhang}, {\em On the computation of jordan canonical
  form}, International Journal of Pure and Applied Mathematics, 78 (2012),
  pp.~155--160.

\end{thebibliography}
	
\end{document}